\documentclass{elsarticle}
\usepackage[english]{babel}
\usepackage{amsmath}
\usepackage{amssymb}
\usepackage{fancyhdr} 
\usepackage{bussproofs}
\usepackage{mathrsfs}
\usepackage{tikz}
\usetikzlibrary{shapes.geometric}

\usepackage{graphicx}
\usepackage{subfigure}
\usepackage{pstricks}
\usepackage{subfigure}
\usepackage{pst-plot}
\usepackage{array}
\usepackage{tikz}
\usetikzlibrary{matrix,arrows}
\usetikzlibrary{arrows,%
                petri,%
                topaths}%

\usepackage{amsthm}
\usepackage{amsmath}%
\usepackage{amsfonts}%
\usepackage{amssymb}%

\usepackage{enumerate}

\newtheorem{theorem}{Theorem}[section]  

\newtheorem{observation}[section]{Observation}
\theoremstyle{definition}
\newtheorem{definition}[theorem]{Definition}

\newtheorem{proposition}{Proposition}[section]

\newtheorem{lemma}{Lemma}[section]

\theoremstyle{definition} 

\newtheorem*{maintheorem*}{Main Theorem}

\journal{Fuzzy Sets and Systems}

\begin{document}

\begin{frontmatter}

\title{Categorical Accommodation of Graded Fuzzy Topological System, Graded Frame and Fuzzy Topological Space with Graded inclusion}


\author[mymainaddress]{Purbita Jana\corref{mycorrespondingauthor}}
\cortext[mycorrespondingauthor]{Corresponding author}
\ead{purbita\_presi@yahoo.co.in}

\author[mysecondaryaddress]{Mihir .K. Chakraborty}
\ead{mihirc4@gmail.com}

\fntext[myfootnote]{The research work of the first author is supported by ``Women Scientists Scheme-
A (WOS-A)" under the File Number: SR/WOS-A/PM-1010/2014.}

\address[mymainaddress]{Department of Pure Mathematics, University of Calcutta, \\35, Ballygunge Circular Road, Ballygunge, Kolkata-700019, West Bengal,  India.}
\address[mysecondaryaddress]{School of Cognitive Science, Jadavpur University,\\188, Raja S. C. Mallick Road, Kolkata-700032, West Bengal, India.}

\begin{abstract}
A detailed study of graded frame, graded fuzzy topological system and fuzzy topological space with graded inclusion is already done in our earlier paper. The notions of graded fuzzy topological system and fuzzy topological space with graded inclusion were obtained via fuzzy geometric logic with graded consequence. As an off shoot the notion of graded frame has been developed. This paper deals with a detailed categorical study of graded frame, graded fuzzy topological system and fuzzy topological space with graded inclusion and their interrelation. 
\end{abstract}

\begin{keyword}
Fuzzy relations \sep Non-classical logic \sep Category theory \sep Topology \sep Algebra \sep Graded frame \sep Graded fuzzy topological system \sep Fuzzy topological space with graded inclusion\sep Fuzzy geometric logic with graded consequence.
\end{keyword}

\end{frontmatter}


\section{Introduction}

The motivation of this work mostly came from the main topic of Vickers's book ``Topology via Logic" \cite{SV}, where he introduced a notion of topological system and indicated its connection with geometric logic. The relationships among topological space, topological system, frame and geometric logic play an important role in the study of topology through logic (geometric logic). Naturally the question ``from which logic fuzzy topology can be studied?" comes in mind. If such a logic is obtained what could be its significance?

To answer these questions, as basic steps, we first introduced some notions of fuzzy topological systems and established the interrelation with appropriate topological spaces and algebraic structures. These relationships were studied in categorical framework \cite{DM1, MP, MP1, MP2, SR, SO, SO1}.

In \cite{PM}, another level of generalisation (i.e., introduction of many-valuedness) has taken place giving rise to graded fuzzy topological system (vide Definition \ref{gftsy}) and fuzzy topological space with graded inclusion (vide Definition \ref{gfts}). It was also required to generalise the notion of frame to graded frame (vide Definition \ref{gf}).

Geometric logic is endowed with an informal observational semantics \cite{SVI}: whether what has been observed does satisfy (match) an assertion or not. In fact, from the stand point of observation, negative and implicational propositions and universal quantification face ontological difficulties. On the other hand arbitrary disjunction needs to be included (cf. \cite{SVI} for an elegant discussion on this issue). Now, observations of facts and assertions about them may corroborate with each other partially. It is a fact of reality and in such cases it is natural to invoke the concept of `satisfiability to some degree of a statement with respect to the observation'. As a result the question whether some assertion follows from some other assertion might not have a crisp answer `yes'/`no'. It is likely that in general the `relation of following' or more technically speaking, the consequence relation turnstile ($\vdash$) may be itself many-valued or graded (vide Definition \ref{validg}). For an introduction to the general theory of graded consequence relation we refer to \cite{MK,MSD}. Thus, we have adopted graded satisfiability as well as graded consequence in \cite{PM}.

A further generalised notions such as fuzzy geometric logic with graded consequence, fuzzy topological spaces with graded inclusion, graded frame and graded fuzzy topological systems came into the picture \cite{PM}. 

The categories of the fuzzy topological spaces with graded inclusion, graded fuzzy topological systems and graded frames are introduced in this paper. On top of that, this paper depicts the transformation of morphisms between the objects which play an interesting role too. Through the categorical study it becomes more clear why the graded inclusion is important in the fuzzy topology to establish the desired connections. 

\subsection*{Fuzzy geometric logic with graded consequence\cite{PM}}
The \textbf{alphabet} of the language $\mathscr{L}$ of fuzzy geometric logic with graded consequence comprises of the connectives $\wedge$, $\bigvee$, the existential quantifier $\exists$, parentheses $)$ and $($ as well as:
\begin{itemize}
\item countably many individual constants $c_1,c_2,\dots$;
\item denumerably many individual variables $x_1,x_2,\dots$;
\item propositional constants $\top$, $\bot$;
\item for each $i>0$, countably many $i$-place predicate symbols $p^i_j$'s, including at least the $2$-place symbol ``$=$" for identity;
\item for each $i>0$, countably many $i$-place function symbols $f^i_j$'s.
\end{itemize}
\begin{definition}[Term]\label{term}
\textbf{Terms} are recursively defined in the usual way.
\begin{itemize}
\item every constant symbol $c_i$ is a term;
\item every variable $x_i$ is a term;
\item if $f_j$ is an $i$-place function symbol, and $t_1,t_2,\dots,t_i$ are terms then\\ $f^i_jt_1t_2\dots t_i$ is a term;
\item nothing else is a term.
\end{itemize}
\end{definition}
\begin{definition}[Geometric formula]\label{wff}
\textbf{Geometric formulae} are recursively defined as follows:
\begin{itemize}
\item $\top$, $\bot$ are geometric formulae;
\item if $p_j$ is an $i$-place predicate symbol, and $t_1,t_2,\dots,t_i$ are terms then $p^i_jt_1t_2\dots t_i$ is a geometric formula;
\item if $t_i$, $t_j$ are terms then $(t_i=t_j)$ is a geometric formula;
\item if $\phi$ and $\psi$ are geometric formulae then $(\phi\wedge\psi)$ is a geometric formula;
\item if $\phi_i$'s ($i\in I$) are geometric formulae then $\bigvee\{\phi_i\}_{i\in I}$ is a geometric formula, when $I=\{1,2\}$ then the above formula is written as $\phi_1\vee \phi_2$;
\item if $\phi$ is a geometric formula and $x_i$ is a variable then $\exists x_i\phi$ is a geometric formula;
\item nothing else is a geometric formula.
\end{itemize}
\end{definition}
\begin{definition}[Interpretation]\label{interpretation}
An \textbf{interpretation} $I$ consists of
\begin{itemize}
\item a set $D$, called the domain of interpretation;
\item an element $I(c_i)\in D$ for each constant $c_i$;
\item a function $I(f^i_j):D^i\longrightarrow D$ for each function symbol $f^i_j$;
\item a fuzzy relation $I(p^i_j):D^i\longrightarrow [0,1]$ for each predicate symbol $p^i_j$ i.e. it is a fuzzy subset of $D^i$.
\end{itemize}
\end{definition}
\begin{definition}[Graded Satisfiability]\label{sat}
Let $s$ be a sequence over $D$. Let $s=(s_1,s_2,\dots)$ be a sequence over $D$ where $s_1,s_2,\dots$ are all elements of $D$. Let $d$ be an element of $D$. Then $s(d/x_i)$ is the result of replacing $i$'th coordinate of $s$ by $d$ i.e., $s(d/x_i)=(s_1,s_2,\dots,s_{i-1},d,s_{i+1},\dots)$. Let $t$ be a term. Then $s$ assigns an element $s(t)$ of $D$ as follows:
\begin{itemize}
\item if $t$ is the constant symbol $c_i$ then $s(c_i)=I(c_i)$;
\item if $t$ is the variable $x_i$ then $s(x_i)=s_i$;
\item if $t$ is the function symbol $f^i_jt_1t_2\dots t_i$ then\\ $s(f^i_jt_1t_2\dots t_i)=I(f^i_j)(s(t_1),s(t_2),\dots,s(t_i))$.
\end{itemize}
Now we define grade of satisfiability of $\phi$ by $s$ written as $gr(s\ \emph{sat}\ \phi)$, where $\phi$ is a geometric formula, as follows:
\begin{itemize}
\item $gr(s\ \emph{sat}\ p^i_jt_1t_2\dots t_i)=I(p^i_j)(s(t_1),s(t_2),\dots,s(t_i))$;
\item $gr(s\ \emph{sat}\ \top)=1$;
\item $gr(s\ \emph{sat}\ \bot)=0$;
\item $gr(s\ \emph{sat}\ t_i=t_j)$ $=\begin{cases}
        1 & \emph{if $s(t_i)=s(t_j)$} \\
        0 & \emph{otherwise};
    \end{cases}$
\item $gr(s\ \emph{sat}\ \phi_1\wedge\phi_2)=gr(s\ \emph{sat}\ \phi_1)\wedge gr(s\ \emph{sat}\ \phi_2)$;
\item $gr(s\ \emph{sat}\ \phi_1\vee\phi_2)=gr(s\ \emph{sat}\ \phi_1)\vee gr(s\ \emph{sat}\ \phi_2)$;
\item $gr(s\ \emph{sat}\ \bigvee\{\phi_i\}_{i\in I})=sup\{gr(s\ \emph{sat}\ \phi_i)\mid i\in I\}$;
\item $gr(s\ \emph{sat}\ \exists x_i\phi)=sup\{gr(s(d/ x_i)\ \emph{sat}\ \phi)\mid d\in D\}$.
\end{itemize}
\end{definition}
 In $[0,1]$, we have used $\wedge$ and $\vee$ to mean min and max respectively - a convention that will be followed throughout.
The expression $\phi\vdash\psi$, where $\phi$, $\psi$ are wffs, is called a sequent. We now define satisfiability of a sequent.
\begin{definition}\label{validg}
1. $gr(s\ \emph{sat}\ \phi\vdash\psi)$=$gr(s\ \emph{sat}\ \phi)\rightarrow gr(s\ \emph{sat}\ \psi)$, where $\rightarrow:[0,1]\times [0,1]\longrightarrow [0,1]$ is the G$\ddot{o}$del arrow defined as follows:\\
$a\rightarrow b$ $=\begin{cases}
        1 & \emph{if $a\leq b$} \\
        b & \emph{if $a>b$},
    \end{cases}$
    
    for $a,b\in [0,1]$.\\
2. $gr(\phi\vdash\psi)=inf_s\{ gr(s\ \emph{sat}\ \phi\vdash \psi)\}$, where $s$ ranges over all sequences over the domain $D$ of interpretation.
\end{definition}
\begin{theorem}\label{5.1g}
Graded sequents satisfy the following properties
\begin{enumerate}
\item 
$gr(\phi\vdash\phi)=1$,
\item
$gr(\phi\vdash\psi)\wedge gr(\psi\vdash\chi)\leq gr(\phi\vdash\chi)$,
\item
(i) $gr(\phi\vdash\top)=1$,\hspace{29pt}  
(ii) $gr(\phi\wedge\psi\vdash\phi)=1$,\\
(iii) $gr(\phi\wedge\psi\vdash\psi)=1$,\hspace{4pt}
(iv) $gr(\phi\vdash\psi)\wedge gr(\phi\vdash\chi)=gr(\phi\vdash\psi\wedge\chi)$,
\item 
(i) $gr(\phi\vdash\bigvee S)=1$ if $\phi\in S$,\\
(ii)$inf_{\phi\in S} \{gr(\phi\vdash\psi)\}\leq gr(\bigvee S\vdash\psi)$,
\item
$gr(\phi\wedge\bigvee S\vdash\bigvee\{\phi\wedge\psi\mid\psi\in S\})=1$,
\item
$gr(\top\vdash (x=x))=1$,
\item
$gr(((x_1,\dots,x_n)=(y_1,\dots,y_n))\wedge\phi\vdash\phi[(y_1,\dots,y_n)/(x_1,\dots,x_n)])$=1,
\item
(i) $gr(\phi\vdash\psi[x/ y])\leq gr(\phi\vdash\exists y\psi)$,\hspace{4pt}
(ii) $gr(\exists y\phi\vdash\psi)\leq gr(\phi[x/ y]\vdash\psi)$,
\item
$gr(\phi\wedge (\exists y)\psi\vdash(\exists y)(\phi\wedge\psi))=1$.
\end{enumerate}
\end{theorem}

\section{Categories: Graded Fuzzy Top, Graded Fuzzy TopSys, Graded Frm and their interrelationships}
\begin{definition}\label{gfts}
Let $X$ be a set, $\tau$ be a collection of fuzzy subsets of $X$ s.t.
\begin{enumerate}
\item $\tilde\emptyset$ , $\tilde X\in \tau$, where $\tilde\emptyset (x)=0$ and $\tilde X (x)=1$, for all $x\in X$;
\item $\tilde T_i\in\tau$ for $i\in I$ imply $\bigcup_{i\in I}\tilde T_i\in \tau$, where $\bigcup_{i\in I}\tilde{T_i}(x)=sup_{i\in I}\{\tilde{T_i}(x)\}$;
\item $\tilde T_1$, $\tilde T_2\in\tau$ imply $ \tilde T_1\cap\tilde T_2\in\tau$, where $(\tilde{T_1}\cap \tilde{T_2})(x)=\tilde{T_1}(x)\wedge \tilde{T_2}(x)$,
\end{enumerate}
and $\subseteq$ be a fuzzy inclusion relation between fuzzy sets, which is defined as follows. For any two fuzzy subsets $\tilde{T_1}, \tilde{T_2}$ of $X$,   $gr(\tilde{T_1}\subseteq{\tilde{T_2}})=inf_{x\in X}\{\tilde{T_1}(x)\rightarrow\tilde{T_2}(x)\}$, where $\rightarrow$ is the G$\ddot{o}$del arrow.

Then $(X,\tau,\subseteq)$ is called a \textbf{fuzzy topological space with graded inclusion}. $(\tau,\subseteq)$ is called a \textbf{fuzzy topology with graded inclusion} over $X$. 
\end{definition}
It is to be noted that we preferred to use the traditional notation $\tilde{A}$ to denote a fuzzy set \cite{MA}.
We list the properties of the members of fuzzy topology with graded inclusion, as propositions, that would be used subsequently. By routine check all the propositions can be verified. \begin{proposition}\label{gt1}
$gr(\tilde{T}\subseteq\tilde{T})=1.$
\end{proposition}
\begin{proposition}\label{gt2}
$gr(\tilde{T_1}\subseteq\tilde{T_2})=1=gr(\tilde{T_2}\subseteq\tilde{T_1})\Rightarrow \tilde{T_1}=\tilde{T_2}$.
\end{proposition}
\begin{proposition}\label{gt3}
$gr(\tilde{T_1}\subseteq\tilde{T_2})\wedge gr(\tilde{T_2}\subseteq\tilde{T_3})\leq gr(\tilde{T_1}\subseteq\tilde{T_3}).$
\end{proposition}
\begin{proposition}\label{gt4}
$gr(\tilde{T_1}\cap\tilde{T_2}\subseteq\tilde{T_1})=1=gr(\tilde{T_1}\cap\tilde{T_2}\subseteq\tilde{T_2})$.
\end{proposition}
\begin{proposition}\label{gt5}
$gr(\tilde{T}\subseteq\tilde{X})=1$.
\end{proposition}
\begin{proposition}\label{gt6}
$gr(\tilde{T_1}\subseteq\tilde{T_2})\wedge gr(\tilde{T_1}\subseteq\tilde{T_3})=gr(\tilde{T_1}\subseteq\tilde{T_2}\cap\tilde{T_3})$.
\end{proposition}
\begin{proposition}\label{gt7}
$gr(\tilde{T_i}\subseteq\bigcup_i\tilde{T_i})=1.$
\end{proposition}
\begin{proposition}\label{gt8}
$inf_{ \tilde{T_i}\in S}\{gr(\tilde{T_i}\subseteq\tilde{T})\}=gr(\bigcup S\subseteq\tilde{T}).$
\end{proposition}
\begin{proposition}\label{gt9}
$gr(\tilde{T}\cap\bigcup_i\tilde{T_i}\subseteq \bigcup_i(\tilde{T}\cap\tilde{T_i}))=1$.
\end{proposition}
\begin{proposition}\label{gt10}
$\tilde{T_1}(x)\wedge gr(\tilde{T_1}\subseteq\tilde{T_2})\leq \tilde{T_2}(x)$, for each $x$.
\end{proposition}
\underline{Note 1}: In \cite{JA}, a notion called localic preordered topological space has been defined which is a 4-tuple $(X,L,\tau ,P)$ where $L$ is a frame, $(X,L,\tau)$ is an $L$-valued fuzzy topological space (localic topological space), $P$ is an $L$-valued fuzzy preorder (reflexive and transitive) on $X$ and $P(x,y)\wedge \tilde{T}(y)\leq \tilde{T}(x)$, for any $x,\ y\in X$ and $\tilde{T}\in\tau$. In our case, we have taken $L$ as $[0,1]$ and usual fuzzy topological space with the exception that the inclusion relation in $\tau$ is defined in terms of a fuzzy implication (G$\ddot{o}$del). That is, the set $X$ has no ordering here but the topology $\tau$ is endowed with a fuzzy ordering relation namely fuzzy inclusion.
\begin{definition}[Graded Frame]\label{gf}
A \textbf{graded frame} is a 5-tuple $(A,\top,\wedge,\bigvee,R)$, where $A$ is a non-empty set,  $\top\in A$, $\wedge$ is a binary operation, $\bigvee$ is an operation on arbitrary subset of $A$, $R$ is a $[0,1]$-valued fuzzy binary relation on $A$ satisfying the following conditions:
\begin{enumerate}
\item $R(a,a)=1$;
\item $R(a,b)=1=R(b,a)\Rightarrow a=b$;
\item $R(a,b)\wedge R(b,c)\leq R(a,c)$;
\item $R(a\wedge b,a)=1=R(a\wedge b,b)$;
\item $R(a,\top)=1$;
\item $R(a,b)\wedge R(a,c)=R(a,b\wedge c)$;
\item $R(a,\bigvee S)=1$ if $a\in S$;
\item $inf\{R(a,b)\mid a\in S\}=R(\bigvee S,b)$;
\item $R(a\wedge\bigvee S,\bigvee\{a\wedge b\mid b\in S\})=1$;
\end{enumerate}
for any $a,\ b,\  c\in A$ and $S\subseteq A$. We will denote graded frame by $(A,R)$. 
\end{definition}
\underline{Note 2}: A graded frame is a localic preordered set \cite{JA} with the value set $[0,1]$ having some extra conditions namely 4-9 of Definition \ref{gf}.
\begin{definition}\label{gftsy}
A \textbf{Graded fuzzy topological system} is a quadruple $(X,\models,A,R)$ consisting of a nonempty set $X$, a graded frame $(A,R)$ and a fuzzy relation $\models$ from $X$ to $A$ such that
\begin{enumerate}
\item $gr(x\models a)\wedge R(a,b)\leq gr(x\models b)$;
\item for any finite subset including null set, $S$, of $A$, $gr(x\models\bigwedge S)=\\ inf\{gr(x\models a)\mid a\in S\}$;
\item for any subset $S$ of $A$, $gr(x\models\bigvee S)= sup\{gr(x\models a)\mid a\in S\}$.
\end{enumerate} 
\end{definition}
\begin{definition}[Spatial]\label{spatial}
A graded fuzzy topological system $(X,\models,A,R)$ is said to be \textbf{spatial}\index{spatial} if and only if (for any $x\in X$, $gr(x\models a)=gr(x\models b)$) imply ($a=b$), for any $a,\ b\in A$.
\end{definition}
\underline{Note 3}: The localic preordered topological system defined in \cite{JA} has a departure from the above notion in that the fuzzy order relation of localic preordered topological system is defined on the set $X$ while in Definition \ref{gftsy} it is defined on the set $A$. That means in \cite{JA}, the fuzzy preorder relation is defined on the object set $X$ while in our case the fuzzy preorder is imposed on the set of properties $A$, additionally we need more conditions on the preorder of the property set to connect the fuzzy topological system with fuzzy geometric logic.
\subsection{Categories}
\subsection*{$\mathbf{Graded\ Fuzzy\ Top}$}
\begin{definition}[{$\mathbf{Graded\ Fuzzy\ Top}$}]\label{Grft}\index{category! Graded Fuzzy Top} 
The category $\mathbf{Graded\ Fuzzy\ Top}$ is defined thus.
\begin{itemize}
\item The objects are fuzzy topological spaces with graded inclusion $(X,\tau,\subseteq)$, $(Y,\tau ',\subseteq)$ etc (c.f. Definition \ref{gfts}).
\item The morphisms\index{Graded Fuzzy Top!morphism} are functions satisfying the following continuity property: If $f:(X,\tau,\subseteq)\longrightarrow (Y,\tau ',\subseteq)$ and $\tilde{T'}\in\tau '$ then $f^{-1}(\tilde{T'})\in\tau$.
\item The identity on $(X,\tau,\subseteq)$ is the identity function. That this is an\\ $\mathbf{Graded\ Fuzzy\ Top}$\index{Graded Fuzzy Top!morphism} morphism can be proved.
\item If $f:(X,\tau,\subseteq)\longrightarrow (Y,\tau ',\subseteq)$ and 
$g:(Y,\tau ',\subseteq)\longrightarrow (Z,\tau '',\subseteq)$ are morphisms in $\mathbf{Graded\ Fuzzy\ Top}$, their 
composition $g\circ f$ is the composition of functions between two sets. It can be verified that $g\circ f$ is a morphism in $\mathbf{Graded\ Fuzzy\ Top}$.
\end{itemize}
\end{definition}
\begin{proposition}\label{ftres}
$gr(\tilde{T_1}\subseteq\tilde{T_2})\leq gr(f(\tilde{T_1})\subseteq f(\tilde{T_2}))$.
\end{proposition}
\begin{proof}
It is known that $gr(\tilde{T_1}\subseteq\tilde{T_2})=inf_{x\in X}\{\tilde{T_1}(x)\rightarrow\tilde{T_2}(x)\}$ and\\
$gr(f(\tilde{T_1})\subseteq f(\tilde{T_2}))=inf_{y\in Y}\{f(\tilde{T_1})(y)\rightarrow f(\tilde{T_2})(y)\}=\\inf_{y\in Y}\{sup_{x\in f^{-1}(y)}\{\tilde{T_1}(x)\}\rightarrow sup_{x\in f^{-1}(y)}\{\tilde{T_2}(x)\}\}$.\\
Now, $sup_{x\in f^{-1}(y)}\{\tilde{T_1}(x)\}\rightarrow sup_{x\in f^{-1}(y)}\{\tilde{T_2}(x)\}=\\inf_{x\in f^{-1}(y)}\{\tilde{T_1}(x)\rightarrow sup_{x\in f^{-1}(y)}\{\tilde{T_2}(x)\}\}$ and $\tilde{T_2}(x)\leq sup_{x\in f^{-1}(y)}\tilde{T_2}(x)$ for any $x\in f^{-1}(y)$.\\
Therefore $\tilde{T_1}(x)\rightarrow \tilde{T_2}(x)\leq\tilde{T_1}(x)\rightarrow sup_{x\in f^{-1}(y)}\{\tilde{T_2}(x)\}$ for any $x\in f^{-1}(y)$.\\
So, $inf_{x\in f^{-1}(y)}\{\tilde{T_1}(x)\rightarrow \tilde{T_2}(x)\}\leq inf_{x\in f^{-1}(y)}\{\tilde{T_1}(x)\rightarrow sup_{x\in f^{-1}(y)}\{\tilde{T_2}(x)\}\}$.\\
Now, $inf_{x\in X}\{\tilde{T_1}(x)\rightarrow \tilde{T_2}(x)\}\leq inf_{x\in f^{-1}(y)}\{\tilde{T_1}(x)\rightarrow \tilde{T_2}(x)\}$ as $f^{-1}(y)\subseteq X$.\\
So, $inf_{x\in X}\{\tilde{T_1}(x)\rightarrow \tilde{T_2}(x)\}\leq inf_{x\in f^{-1}(y)}\{\tilde{T_1}(x)\rightarrow sup_{x\in f^{-1}(y)}\{\tilde{T_2}(x)\}\}$ for any $y\in Y$ and consequently,\\ $inf_{x\in X}\{\tilde{T_1}(x)\rightarrow \tilde{T_2}(x)\}\leq inf_{y\in Y}\{sup_{x\in f^{-1}(y)}\{\tilde{T_1}(x)\}\rightarrow sup_{x\in f^{-1}(y)}\{\tilde{T_2}(x)\}\}$.
\end{proof}
\begin{definition}[{$\mathbf{Graded\ Frm}$}]\label{Grfrm} 
The category $\mathbf{Graded\ Frm}$\index{category!Graded Frm} is defined thus.
\begin{itemize}
\item The objects are graded frames $(A,R)$, $(B,R')$ etc. (c.f. Definition \ref{gf}).
\item The morphisms\index{Graded Frm!morphism} are graded frame homomorphisms\index{graded!frame homomorphism} defined in the following way: If $f:(A,R)\longrightarrow (B,R')$ then\\
(i) $f(a_1\wedge a_2)=f(a_1)\wedge f(a_2)$;\\
(ii) $f(\bigvee_i a_i)=sup_i\{f(a_i)\}$;\\
(iii) $R(a_1,a_2)\leq R'(f(a_1),f(a_2))$.
\item The identity on $(A,R)$ is the identity morphism. That this is an $\mathbf{Graded\ Frm}$ morphism\index{Graded Frm!morphism} can be proved.
\item If $f:(A,R)\longrightarrow (B,R')$ and 
$g:(B,R')\longrightarrow (C,R'')$ are morphisms in $\mathbf{Graded\ Frm}$, their 
composition $g\circ f$ is the composition of graded homomorphisms between two graded frames. It can be verified that 
$g\circ f$ is a morphism in $\mathbf{Graded\ Frm}$ (vide Proposition \ref{next}).
\end{itemize}
\end{definition}
\begin{proposition}\label{next}
If $f:(A,R)\longrightarrow (B,R')$ and 
$g:(B,R')\longrightarrow (C,R'')$ are morphisms in $\mathbf{Graded\ Frm}$ then $g\circ f:(A,R)\longrightarrow (C,R'')$ is a morphism in $\mathbf{Graded\ Frm}$.
\end{proposition}
\begin{proof}
To show that $g\circ f:(A,R)\longrightarrow (C,R'')$ is a graded frame homomorphism we will proceed in the following way.
\begin{align*}
(i)\ g\circ f (a_1\wedge a_2) & = g(f(a_1\wedge a_2))\\
& = g(f(a_1)\wedge f(a_2)),\ \text{as $f$ is a graded frame homomorphism}\\
& = g(f(a_1))\wedge g(f(a_2)),\text{as $g$ is a graded frame homomorphism}\\
& = g\circ f(a_1)\wedge g\circ f(a_2).\\
(ii)\ g\circ f (\bigvee_i a_i) & = g(f(\bigvee_i a_i))\\
& = g(sup_i\{f(a_i)\}),\ \text{as $f$ is a graded frame homomorphism}\\
& = sup_i\{g(f(a_i))\},\text{as $g$ is a graded frame homomorphism}\\
& = sup_i\{g\circ f(a_i)\}.\\
(iii)\ R(a_1,a_2) & \leq R'(f(a_1),f(a_2)),\ \text{as $f$ is a graded frame homomorphism}\\
& \leq R''(g(f(a_1)),g(f(a_2))),\text{as $g$ is a graded frame homomorphism}\\
& = R''(g\circ f(a_1),g\circ f(a_2)).
\end{align*}
\end{proof}
\begin{definition}[{$\mathbf{Graded\ Fuzzy\ TopSys}$}]\label{Grftsy} 
The category $\mathbf{Graded\ Fuzzy\ TopSys}$\index{category!Graded Fuzzy TopSys} is defined thus.
\begin{itemize}
\item The objects are graded fuzzy topological systems $(X,\models,A,R)$, $(Y,\models,B,R')$ etc. (c.f. Definition \ref{gftsy}).
\item The morphisms are pair of maps satisfying the following continuity properties: If $(f_1,f_2):(X,\models ,A,R)\longrightarrow (Y,\models',B,R')$ then\\
(i) $f_1:X\longrightarrow Y$ is a set map;\\
(ii) $f_2:(B,R')\longrightarrow (A,R)$ is a graded frame homomorphism;\\
(iii) $gr(x\models f_2(b))=gr(f_1(x)\models' b)$.
\item The identity on $(X,\models,A,R)$ is the pair $(id_X,id_A)$, where $id_X$ is the identity map on $X$ and $id_A$ is the identity graded frame homomorphism. That this is an $\mathbf{Graded\ Fuzzy\ TopSys}$ morphism\index{Graded Fuzzy TopSys!morphism} can be proved.
\item If $(f_1,f_2):(X,\models,A,R)\longrightarrow (Y,\models',B,R')$ and 
$(g_1,g_2):(Y,\models',B,R')\longrightarrow (Z,\models'',C,R'')$ are morphisms in $\mathbf{Graded\ Fuzzy\ TopSys}$, their 
composition $(g_1,g_2)\circ (f_1,f_2)=(g_1\circ f_1,f_2\circ g_2)$ is the pair of composition of functions between two sets and composition of graded homomorphisms between two graded frames. It can be verified that 
$(g_1,g_2)\circ (f_1,f_2)$ is a morphism in $\mathbf{Graded\ Fuzzy\ TopSys}$.
\end{itemize}
\end{definition}
\subsection{Functors}
In this subsection, we define various functors, which are required to prove our desired results.
\subsection*{Functor $Ext_g$ from Graded Fuzzy TopSys to Graded Fuzzy Top}\index{functor!$Ext_g$}
\begin{definition}\label{3.8g}
Let $(X,\models,A,R)$ be a graded fuzzy topological system, and $a\in A$. For each $a$, its \textbf{extent$_g$}\index{extent$_g$} in $(X,\models, A,R)$ is a mapping $ext_g(a)$ from $X$ to $[0,1]$ given by $ext_g(a)(x)=gr(x\models a)$.
Also $ext_g(A)=\{ext_g(a)\}_{a\in A}$ and $gr(ext_g(a_1)\subseteq ext_g(a_2))=inf\{ext_g(a_1)(x)\rightarrow ext_g(a_2)(x)\}_x$, for any $a_1,a_2\in A$.
\end{definition}
\begin{lemma}\label{3.9g}
$(ext_g(A),\subseteq)$ forms a graded fuzzy topology on $X$.
\end{lemma}
As a consequence $(X,ext_g(A),\subseteq)$ forms a graded fuzzy topological space.
\begin{lemma}\label{3.10g}
If $(f_1,f_2):(X,\models ',A,R)\longrightarrow (Y,\models '',B,R')$ is continuous then \\$f_1:(X,ext_g(A),\subseteq)\longrightarrow (Y,ext_g(B),\subseteq)$ is continuous.
\end{lemma}
\begin{proof}
$(f_1,f_2):(X,\models ',A,R)\longrightarrow (Y,\models '',B,R')$ is continuous.\\
So we have for all $x\in X, b\in B$, $gr(x\models 'f_2(b))=gr(f_1(x)\models ''b)$.
\begin{align*} 
Now, (f_1^{-1}(ext_g(b)))(x) & =ext_g(b)(f_1(x))\\
& =gr(f_1(x)\models ''b)\\
& =gr(x\models 'f_2(b))\\
& =ext_g(f_2(b))(x).
\end{align*}
$So,$ $f_1^{-1}(ext_g(b))=ext_g(f_2(b))\in ext_g(A)$.\\
$So,$ $f_1$ is a continuous map from $(X,ext_g(A),\subseteq)$ to $(Y,ext_g(B),\subseteq)$.
\end{proof}
\begin{definition}\label{3.11g}
$\mathbf{Ext_g}$\index{functor!$Ext_g$} is a functor from $\mathbf{Graded\ Fuzzy\ TopSys}$ to\\ $\mathbf{Graded\ Fuzzy\ Top}$ defined as follows.\\
$Ext_g$ acts on an object $(X,\models ',A,R)$ as $Ext_g(X,\models ',A,R)=(X,ext_g(A),\subseteq)$ and on a morphism $(f_1,f_2)$ as $Ext_g(f_1,f_2)=f_1$.
\end{definition}
The above two Lemmas \ref{3.9g} and \ref{3.10g} show that $Ext_g$ is a functor.
\subsection*{Functor $J_g$ from $\mathbf{Graded\ Fuzzy\ Top}$ to $\mathbf{Graded\ Fuzzy\ TopSys}$}
\begin{definition}\label{3.12g}
$\mathbf{J_g}$\index{functor!$J_g$} is a functor from $\mathbf{Graded\ Fuzzy\ Top}$ to\\ $\mathbf{Graded\ Fuzzy\ TopSys}$ defined as follows.\\
$J_g$ acts on an object $(X,\tau,\subseteq)$ as $J_g(X,\tau,\subseteq)=(X, \in ,\tau,\subseteq)$, where $gr(x\in \tilde{T})=\tilde{T}(x)$ for $\tilde{T}$ in $\tau$, and on a morphism $f$ as $J_g(f)=(f,f^{-1})$.
\end{definition}
\begin{lemma}\label{3.13g}
$(X,\in , \tau,\subseteq)$ is a graded fuzzy topological system.
\end{lemma}
\begin{proof}
$(\tau,\subseteq)$ forms a graded frame can be shown using Propositions \ref{gt1}-\ref{gt9}.\\
To show
(1) $gr(x\in \tilde{T_1})\wedge gr(\tilde{T_1}\subseteq\tilde{T_2})\leq gr(x\in \tilde{T_2})$,
(2) $gr(x\in \tilde{T_1}\cap \tilde{T_2})=inf\{ gr(x\in \tilde{T_1}),gr(x\in \tilde{T_2})\}$ and (3) $gr(x\in \bigcup_i \tilde{T_i})=sup_i\{ gr(x\in \tilde{T_i})\}$ let us proceed in the following way.\\
(1) $gr(x\in \tilde{T_1})\wedge gr(\tilde{T_1}\subseteq\tilde{T_2})\leq gr(x\in \tilde{T_2})$ follows by Proposition \ref{gt10}.
\begin{align*}
(2)\ gr(x\in \tilde{T_1}\cap \tilde{T_2}) & =(\tilde{T_1}\cap \tilde{T_2})(x)\\
& =\tilde{T_1}(x)\wedge \tilde{T_2}(x)\\
& = gr(x\in \tilde{T_1})\wedge gr(x\in \tilde{T_2}).\\
(3)\ gr(x\in \bigcup_i \tilde{T_i}) & =(\bigcup_i \tilde{T_i})(x)\\
& =sup_i\{\tilde{T_i}(x)\}\\
& =sup_i\{ gr(x\in \tilde{T_i})\}
\end{align*}
\end{proof}
\begin{lemma}\label{3.14g}
$J_g(f)=(f,f^{-1})$ is continuous provided $f$ is continuous.
\end{lemma}
\begin{proof}
We have $f:(X,\tau_1)\longrightarrow (Y,\tau_2)$ and $(f,f^{-1}):(X,\in ,\tau_1)\longrightarrow (Y,\in ,\tau_2)$.
\\ It is enough to show that $\tilde{T_2}\in \tau_2$, $gr(x\in f^{-1}(\tilde{T_2}))=gr(f(x)\in \tilde{T_2})$ as Proposition \ref{ftres} holds.\\
Now, $gr(x\in f^{-1}(\tilde{T_2}))=(f^{-1}(\tilde{T_2}))(x)=\tilde{T_2}(f(x))=gr(f(x)\in \tilde{T_2})$.\\
Hence $J_g(f)=(f,f^{-1})$ is continuous.
\end{proof}
So $ J_g$ is a functor from $\mathbf{Graded\ Fuzzy\ Top}$ to $\mathbf{Graded\ Fuzzy\ TopSys}$.
\subsection*{Functor $fm_g$ from $\mathbf{Graded\ Fuzzy\ TopSys}$ to $\mathbf{Graded\ Frm^{op}}$}
\begin{definition}\label{3.15g}
$\mathbf{fm_g}$\index{functor!$J_g$} is a functor from $\mathbf{Graded\ Fuzzy\ TopSys}$ to\\ $\mathbf{Graded\ Frm^{op}}$ defined as follows.\\
$fm_g$ acts on an object $(X,\models ,A,R)$ as $fm_g(X,\models ,A,R)=(A,R)$ and on a morphism $(f_1,f_2)$ as $fm_g(f_1,f_2)=f_2$.
\end{definition}
 It is easy to see that $fm_g$ is a functor.
\subsection*{Functor $S_g$ from $\mathbf{Graded\ Frm^{op}}$ to $\mathbf{Graded\ Fuzzy\ TopSys}$}
\begin{definition}\label{3.16g}
Let $(A,R)$ be a graded frame, $Hom_g((A,R),([0,1],R^*))=\{ graded\ frame$ $ hom$ $v:(A,R)\longrightarrow ([0,1],R^*)\}$, where $R^*:[0,1]\times [0,1]\longrightarrow [0,1]$ such that $R^*(a,b)=1$ iff $a\leq b$ and $R^*(a,b)=b$ iff $a>b$, that is, G$\ddot{o}$del arrow.
\end{definition}
\begin{lemma}\label{3.17g}
$(Hom_g((A,R),([0,1],R^*)),\models_*,A,R)$, where $(A,R)$ is a graded frame and $gr(v\models_* a)=v(a)$, is a graded fuzzy topological system.
\end{lemma}
\begin{proof}
First of all, we need to show that $gr(v\models_*a)\wedge R(a,b)\leq gr(v\models_*b)$. That is, it will be enough to show that $v(a)\wedge R(a,b)\leq v(b)$. Now as $v$ is a graded frame homomorphism so $R(a,b)\leq R^*(v(a),v(b))$. Hence $v(a)\wedge R(a,b)\leq R(a,b)\leq R^*(v(a),v(b))$. Now as $v(a)$, $v(b)\in [0,1]$ so either $v(a)\leq v(b)$ or $v(a)>v(b)$. For $v(a)\leq v(b)$, $v(a)\wedge R(a,b)\leq v(a)\leq v(b)$ and for $v(a)>v(b)$, $R^*(v(a),v(b))=0$. So for the second case $v(a)\wedge R(a,b)\leq R^*(v(a),v(b))=v(b)$. Hence $gr(v\models_*a)\wedge R(a,b)\leq gr(v\models_*b)$.\\
As $v$ is a graded frame homomorphism, $v(a\wedge b)=v(a)\wedge v(b)$ and hence $gr(v\models_*a\wedge b)=gr(v\models_*a)\wedge gr(v\models_*b)$.\\
Similarly, $gr(v\models_*\bigvee S)=sup\{gr(v\models_* a)\mid a\in S\}$, for $S\subseteq A$.
\end{proof}
\begin{lemma}\label{3.18g}
If $f:(B,R')\longrightarrow (A,R)$ is a graded frame homomorphism then \\$(\_\circ f,f):(Hom((A,R),([0,1],R^*)),\models_*,A,R)\longrightarrow (Hom((B,R'),([0,1],R^*)),\models_*,B,R')$ is continuous.
\end{lemma}
\begin{definition}\label{3.19g}
$\mathbf{S_g}$\index{functor!$S_g$} is a functor from $\mathbf{Graded\ Frm^{op}}$ to $\mathbf{Graded\ Fuzzy\ TopSys}$ defined as follows.\\
$S_g$ acts on an object $(A,R)$ as $S_g((A,R))=(Hom_g((A,R),([0,1],R^*)),\models_*,A,R)$ and on a morphism $f$ as $S_g(f)=(\_\circ f,f)$.
\end{definition}
Previous two Lemmas \ref{3.17g} and \ref{3.18g} show that $S_g$ is indeed a functor.
\begin{lemma}\label{3.20g}
$Ext_g$ is the right adjoint to the functor $J_g$.
\end{lemma}
\begin{proof}[Proof Sketch]
It is possible to prove the theorem by presenting the co-unit of the adjunction.\\
Recall that $J_g(X,\tau,\subseteq)=(X,\in, \tau,\subseteq)$ and $Ext_g(X,\models,A,R)=(X,ext_g(A),\subseteq)$.\\
$So, J_g(Ext_g(X,\models,A,R))=(X,\in,ext(A),\subseteq)$.\\
Let us draw the diagram of co-unit
\begin{center}
\begin{tabular}{ l | r } 
$\mathbf{Graded\ Fuzzy\ TopSys}$ & $\mathbf{Graded\ Fuzzy\ Top}$\\
\hline
 {\begin{tikzpicture}[description/.style={fill=white,inner sep=2pt}] 
    \matrix (m) [matrix of math nodes, row sep=2.5em, column sep=2.5em]
    { J_g(Ext_g(X,\models,A,R))&&(X,\models,A,R)  \\
         J_g(Y,\tau ',\subseteq) \\ }; 
    \path[->,font=\scriptsize] 
        (m-1-1) edge node[auto] {$\xi_X$} (m-1-3)
        (m-2-1) edge node[auto] {$J_g(f)(\equiv(f_1,f_1^{-1}))$} (m-1-1)
        (m-2-1) edge node[auto,swap] {$\hat f (\equiv(f_1,f_2))$} (m-1-3)
         ;
\end{tikzpicture}} & {\begin{tikzpicture}[description/.style={fill=white,inner sep=2pt}] 
    \matrix (m) [matrix of math nodes, row sep=2.5em, column sep=2.5em]
    { Ext_g(X,\models,A,R)  \\
         (Y,\tau ',\subseteq) \\ }; 
    \path[->,font=\scriptsize] 
        (m-2-1) edge node[auto,swap] {$f(\equiv f_1)$} (m-1-1)
       
         ;
\end{tikzpicture}} \\ 
\end{tabular}
\end{center}
Hence co-unit is defined by $\xi_X=(id_X,ext_g^*)$
\begin{center}
\begin{tikzpicture}[description/.style={fill=white,inner sep=2pt}] 
    \matrix (m) [matrix of math nodes, row sep=2.5em, column sep=2.5em]
    { i.e.\ (X,\in ,ext_g(A),\subseteq )&&(X,\models,A,R )  \\
          }; 
    \path[->,font=\scriptsize] 
        (m-1-1) edge node[auto] {$\xi_X$} (m-1-3)
        (m-1-1) edge node[auto,swap] {$(id_X,ext_g^*)$} (m-1-3)
         ;
\end{tikzpicture}
\end{center}
 where $ext_g^*$ is a mapping from $A$ to $ext_g(A)$ such that $ext_g^*(a)=ext_g(a)$, for all $a\in A$.
 
Now by routine check one can conclude that $Ext_g$ is the right adjoint to the functor $J_g$.
\end{proof}
Diagram of the unit of the above adjunction is as follows:

\begin{center}
\begin{tabular}{ l | r  } 
 $\mathbf{Graded\ Fuzzy\ Top}$ & $\mathbf{Graded\ Fuzzy\ TopSys}$ \\
\hline
 {\begin{tikzpicture}[description/.style={fill=white,inner sep=2pt}] 
    \matrix (m) [matrix of math nodes, row sep=2.5em, column sep=2.5em]
    {(X,\tau,\subseteq) & &Ext(J(X, \tau,\subseteq))  \\
        & & Ext(Y,\models' ,B,R') \\ }; 
    \path[->,font=\scriptsize] 
        (m-1-1) edge node[auto] {$\eta_X(\equiv id_X)$} (m-1-3)
        (m-1-1) edge node[auto,swap] {$\hat{f}(\equiv f_1)$} (m-2-3)
        (m-1-3) edge node[auto] {$ext(f)(\equiv f_1)$} (m-2-3);
\end{tikzpicture}} &  {\begin{tikzpicture}[description/.style={fill=white,inner sep=2pt}] 
    \matrix (m) [matrix of math nodes, row sep=2.5em, column sep=2.5em]
    {& &J(X,\tau,\subseteq)  \\
        & &(Y,\models' ,B,R') \\ }; 
    \path[->,font=\scriptsize]

        (m-1-3) edge node[auto] {$f(\equiv(f_1,f_1^{-1}))$} (m-2-3);
\end{tikzpicture}}\\
\end{tabular}
\end{center}

\begin{observation}\label{o1_g}
If a graded fuzzy topological system $(X,\models,A,R)$ is spatial then the co-unit $\xi_X$ becomes a natural isomorphism.
\end{observation}
\begin{observation}\label{o2_g}
For any graded fuzzy topological space $(X,\tau,\subseteq)$, the unit $\eta_X$ is a natural isomorphism.  
\end{observation}
Observation \ref{o1_g} and Observation \ref{o2_g} gives the following theorem.
\begin{theorem}\label{equigft}
Category of spatial graded fuzzy topological systems is equivalent to the category $\mathbf{Graded\ Fuzzy\ Top}$.
\end{theorem}
\begin{lemma}\label{3.22g}
$fm_g$ is the left adjoint to the functor $S_g$.
\end{lemma}
\begin{proof}[Proof Sketch]
It is possible to prove the theorem by presenting the unit of the adjunction.\\
Recall that  $S_g((B,R'))=(Hom_g((B,R'),([0,1],R^*)),\models_*,B,R')$ where, $gr(v\models_* a)=v(a)$.\\
Hence $S_g(fm_g(X,\models ,A,R))=(Hom_g((A,R),([0,1],R^*)),\models_* ,A,R)$.
\begin{center}
\begin{tabular}{ l | r  } 
 $\mathbf{Graded\ Fuzzy\ TopSys}$ & $\mathbf{Graded\ Frm^{op}}$ \\
\hline
 {\begin{tikzpicture}[description/.style={fill=white,inner sep=2pt}] 
    \matrix (m) [matrix of math nodes, row sep=2.5em, column sep=2.5em]
    {(X,\models ,A,R) & &S_g(fm_g(X, \models ,A,R))  \\
        & & S_g((B,R')) \\ }; 
    \path[->,font=\scriptsize] 
        (m-1-1) edge node[auto] {$\eta_A$} (m-1-3)
        (m-1-1) edge node[auto,swap] {$f(\equiv(f_1,f_2))$} (m-2-3)
        (m-1-3) edge node[auto] {$S_g\hat{f}$} (m-2-3);
\end{tikzpicture}} &  {\begin{tikzpicture}[description/.style={fill=white,inner sep=2pt}] 
    \matrix (m) [matrix of math nodes, row sep=2.5em, column sep=2.5em]
    {& &fm_g(X,\models ,A,R)  \\
        & &(B,R') \\ }; 
    \path[->,font=\scriptsize]

        (m-1-3) edge node[auto] {$\hat{f}(\equiv f_2)$} (m-2-3);
\end{tikzpicture}}\\
\end{tabular}
\end{center}
Then unit is defined by $\eta_A =(p^*,id_A)$.\\
 \begin{tikzpicture}[description/.style={fill=white,inner sep=2pt}] 
    \matrix (m) [matrix of math nodes, row sep=2.5em, column sep=2.5em]
    {i.e.\ (X,\models ,A,R) & &S_g(fm_g(X, \models ,A,R))  \\
        }; 
    \path[->,font=\scriptsize] 
        (m-1-1) edge node[auto] {$\eta_A$} (m-1-3)
        (m-1-1) edge node[auto,swap] {$(p^*,id_A)$} (m-1-3)
        ;
\end{tikzpicture}

where 
$\\p^*:X\longrightarrow Hom_g((A,R),([0,1],R^*))$,

$x\longmapsto p_x:(A,R)\longrightarrow ([0,1],R^*)$
such that $p_x(a)=gr(x\models a)$ and $R(a,a')\leq R^*(p_x(a),p_x(a'))$ for any $a,\ a'\in A$.
Now by routine check one can conclude that $fm_g$ is the left adjoint to the functor $S_g$.
\end{proof}
\begin{theorem}\label{3.25g}
$Ext_g\circ S_g$ is the right adjoint to the functor $fm_g\circ J_g$.
\end{theorem}
\begin{proof}
Follows from the combination of the adjoint situations in Lemmas \ref{3.20g}, \ref{3.22g}.
\end{proof}
The obtained functorial relationships can be illustrated by the following diagram:  
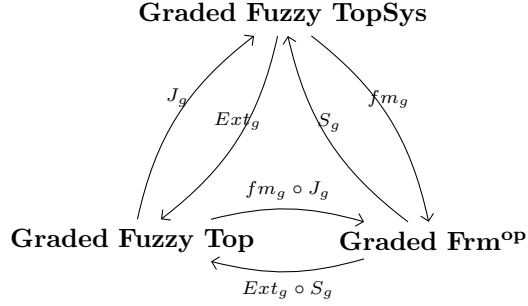
\begin{center}
\begin{tikzpicture}
\node (C) at (0,3) {$\mathbf{Graded\ Fuzzy\ TopSys}$};
\node (A) at (-2,0) {$\mathbf{Graded\ Fuzzy\ Top}$};
\node (B) at (2,0) {$\mathbf{Graded\ Frm^{op}}$};
\path[->,font=\scriptsize ,>=angle 90]
(A) edge [bend left=15] node[above] {$fm_g\circ J_g$} (B);
\path[<-,font=\scriptsize ,>=angle 90]
(A)edge [bend right=15] node[below] {$Ext_g \circ S_g$} (B);
\path[->,font=\scriptsize ,>=angle 90]
(A) edge [bend left=20] node[above] {$J_g$} (C);
\path[<-,font=\scriptsize ,>=angle 90]
(A)edge [bend right=20] node[above] {$Ext_g$} (C);
\path[->,font=\scriptsize, >=angle 90]
(C) edge [bend left=20] node[above] {$fm_g$} (B);
\path[<-,font=\scriptsize, >=angle 90]
(C)edge [bend right=20] node[above] {$S_g$} (B);
\end{tikzpicture}
\end{center}
\section{Concluding Remarks and Future Directions}
\begin{itemize}
\item This work can be considered as a continuation of the work done in \cite{PM}. The objects of the categories, considered here, played the role of model of fuzzy geometric logic with graded consequence. These concepts are actually the byproduct of the study of fuzzy topology via fuzzy geometric logic with graded consequence \cite{PM}.
\item All the results of this paper may be obtained if we consider the value set as totally ordered frame instead of $[0,1]$.
\item It should be noted that the inclusion relation is defined using G$\ddot{o}$del arrow. Hence there is a space to change the arrow to define graded inclusion and make appropriate graded algebraic structures to get the similar connections of this work. Although we do not delve into this matter in this work but hope to deal with this in near future. 
 \item Considering the value set as any frame will give very interesting results such as,
 \begin{enumerate}
 \item For graded fuzzy topological space with graded inclusion if we take the value set any frame $L$ instead of $[0,1]$ and define inclusion relation using the fuzzy arrow viz. $\rightarrow:L\times L\longrightarrow L$ such that \\
 $a\rightarrow b$ $=\begin{cases}
        1_L & \emph{if $a\leq b$} \\
        b & \emph{otherwise},
    \end{cases}$
    
    for $a,b\in L$, where $1_L$ is the top element of $L$,\\
then Proposition \ref{gt6} does not hold. But note that all other propositions shall hold good. The fuzzy arrow defined here will be denoted by G$\ddot{o}$del like arrow over $L$.
\item In case of graded frame if we consider the relation $R$, an $L$-valued fuzzy relation, together with all the conditions other than $R(a,b)\wedge R(a,c)=R(a,b\wedge c)$ and include the condition $R(a,b)\wedge R(a,c)\geq R(a,b\wedge c)$, then we may get the categorical connection with the above described space.
\item It is to be noted that if we consider $L$-fuzzy subset ($L$ is any frame) of $D^i$, where $D$ is the domain of interpretation, as the interpretation of each $i$-ary predicate symbols and define the validity of a sequent using G$\ddot{o}$del like arrow over $L$ then the introduction rule for conjunction will be not valid. 
 \end{enumerate}
 Note that the mathematical structures involved here will become models of fuzzy geometric logic with graded consequence without the introduction rule for conjunction. Fuzzy geometric logic without the introduction rule for conjunction is itself an interesting issue to study. The work in this direction is in our future agenda.
\end{itemize}
\paragraph{Acknowledgements} 
The research work of the first author is supported by ``Women Scientists Scheme-
A (WOS-A)" under the File Number: SR/WOS-A/PM-1010/2014. This research was also supported by the \emph{Indo-European Research Training Network in Logic} (IERTNiL) funded by the \emph{Institute of Mathematical Sciences, Chennai}, the \emph{Institute for Logic, Language and Computation} of the \emph{Universiteit van Amsterdam} and the
\emph{Fakult\"at f\"ur Mathematik, Informatik und Naturwissenschaften} of the \emph{Universit\"at Hamburg}.


\bibliography{mybibfile}


\end{document}